\documentclass{article}
\usepackage{latexsym}
\usepackage{mathrsfs}
\usepackage{indentfirst}
\usepackage{amsmath,amsthm, amssymb}
\usepackage[dvips]{graphicx}
\usepackage{graphicx}
\usepackage{epsfig}
\usepackage{verbatim}

\newtheorem{theorem}{Theorem}[section]
\newtheorem{lemma}[theorem]{Lemma}

\begin{document}
\title{Bipartite graphs with a perfect matching and digraphs \thanks{Work supported by the
Scientific Research Foundation of Guangdong Industry Technical
College, granted No. 2005-11.}}
\author{Zan-Bo Zhang$^1$$^2$\thanks{Corresponding Author. Email address: eltonzhang2001@yahoo.com.cn.}, Dingjun
Lou$^2$
\\ $^1$Department of Computer Engineering,
\\ Guangdong Industry Technical College, Guangzhou 510300, China
\\ $^2$Department of Computer Science,
\\ Sun Yat-sen University, Guangzhou 510275, China
}
\date{}
\maketitle

\begin{abstract}
In this paper, we introduce a corresponding between bipartite
graphs with a perfect matching and digraphs, which implicates an
equivalent relation between the extendibility of bipartite graphs
and the strongly connectivity of digraphs. Such an equivalent
relation explains the similar results on $k$-extendable bipartite
graphs and $k$-strong digraphs. We also study the relation among
$k$-extendable bipartite graphs, $k$-strong digraphs and
combinatorial matrices. For bipartite graphs that are not
1-extendable and digraphs that are not strong, we prove that the
elementary components and strong components are counterparts.
$\newline$\noindent\textbf{Key words}: $k$-extendable, strongly
$k$-connected, indecomposable, irreducible, strong component,
elementary component
\end{abstract}
\section{Introduction and terminologies}

In this paper, all graphs (digraphs) considered have no loop and
multiple edge (arc) unless explicitly stated. For all
terminologies not defined, we refer the reader to \cite{BG1},
\cite{B1} and \cite{LP1}. All matrices considered are zero-one
matrices.

We use $V(G)$ and $E(G)$ to denote the vertex set and edge set of
a graph $G$. Let $G$ be a bipartite graph with bipartition $(U,W)$
where $U=\{u_1,\ldots,u_n\}$ and $W=\{w_1,\ldots,w_n\}$. The
matrix $A=(a_{ij})_{n\times n}$, where $a_{ij}=1$ if and only if
$u_i w_j\in E(G)$, is called the \emph{reduced adjacency matrix}
of $G$. We denote $A$ by $R(G)$. We call $G$ the \emph{reduced
associated bipartite graph} of $A$ and denote $G$ by $B(A)$.

A connected graph is \emph{elementary} if the union of its perfect
matchings forms a connected subgraph. A connected graph $G$ is
called \emph{$k$-extendable}, for $k\leq (|V(G)|-1)/2$, if $G$ has
a matching of size $k$, and every matching of size $k$ of $G$ is
contained in a perfect matching of $G$. $G$ is said to be
\emph{minimal $k$-extendable} if $G$ is $k$-extendable but $G-e$
is not $k$-extendable for any $e\in E(G)$. An edge of $G$ is
called a \emph{fixed single} (\emph{fixed double}) edge if it
belongs to no (all) perfect matchings of $G$. An edge of $G$ is
called \emph{fixed} if it is either a fixed single or a fixed
double edge of $G$. All non-fixed edges of $G$ form a subgraph
$H$, each component of which is elementary and is therefore called
an \emph{elementary component}.

Let $D$ be a digraph. We denote by $V(D)$, $A(D)$ and $M(D)$ the
vertex set, arc set and the adjacent matrix of $D$. Let $M$ be an
adjacent matrix of $D$, we call $D$ the \emph{associated digraph}
of $M$ and denote $D$ by $D(M)$. $D$ is \emph{strongly connected},
or \emph{strong}, if there exists a path from $x$ to $y$ and a
path from $y$ to $x$ in $D$ for any $x,y\in V(D)$, $x\neq y$. A
set $S \subset V(D)$ is a \emph{separator} if $D-S$ is not strong.
$D$ is \emph{k-strongly connected}, or \emph{k-strong}, if
$|V(D)|\geq k+1$ and $D$ has no separator of order less than $k$.
$D$ is \emph{minimal $k$-strong} if $D$ is $k$-strong, but $D-a$
is not $k$-strong for any arc $a\in A(D)$. A \emph{strong
component} is a maximal subdigraph of $D$ which is strong.

We call a path, directed or undirected, from $u$ to $v$ a
$(u,v)$-path. The set of the end-vertices of the edges in a
matching $M$ is denoted by $V(M)$, or $V(e)$ if $M=\{e\}$. The
symmetric difference of two sets $S_1$ and $S_2$, is denoted by
$S_1\vartriangle S_2$. $\newline$

Let $B_n$ denote the set of all matrices of order $n$ over the
Boolean algebra $\{0,1\}$. We call a matrix $A\in B_n$
\emph{reducible} if there exists a permutation matrix $P$, such
that $$P^{T}AP=\left[
\begin{matrix} B & 0 \cr C & D \end{matrix} \right],$$ where $B$
is an $l\times l$ matrix and $D$ is an $(n-l)\times(n-l)$ matrix,
for some $1\leq l \leq n-1$. $A$ is \emph{irreducible} if it is
not reducible. Let $k$ be an integer with $1 \leq k \leq n$. $A$
is called \emph{$k$-reducible} if there exists a permutation
matrix $P$, such that $$P^{T}AP=\left[\begin{matrix} A_{11} &
A_{12} & 0 \cr A_{21} & A_{22} & A_{23} \end{matrix} \right], $$
where $A_{11}$ and $[A_{22} \ \ A_{23} ]$ are square matrices of
order at least one and the size of the zero submatrix at the upper
right corner is $l\times (n-k+1-l)$, $1\leq l \leq n-1$. If $A$ is
not $k$-reducible, then $A$ is called \emph{$k$-irreducible}.

A matrix $A\in B_n$ is call \emph{partly decomposable} if there
exist permutation matrices $P$ and $Q$, such that $$PAQ=\left[
\begin{matrix} B & 0 \cr C & D \end{matrix} \right],$$ where $B$ is an $l\times l$
matrix and $D$ is an $(n-l)\times (n-l)$ matrix, for some $1\leq l
\leq n-1$. $A$ is \emph{fully indecomposable} if it is not partly
decomposable. Let $k$ be an integer with $0 \leq k \leq n$. $A$ is
called \emph{$k$-partly decomposable} if it contains an $l\times
(n-k+1-l)$ zero submatrix, for some $1\leq l \leq n-1$. A matrix
which is not $k$-partly decomposable is called
\emph{$k$-indecomposable}.

A \emph{diagonal} of a matrix $A=(a_{ij})\in B_n$ is a collection
$T$ of $n$ entries $a_{1i_1},a_{2i_2},\ldots,a_{ni_n}$ of $A$ such
that $\{i_1,i_2,\ldots, i_n\}=\{1,2,\ldots,n\}$. If $i_j=j$ for
$j=1$, 2, $\ldots$, n, we call the diagonal \emph{main diagonal}
of the matrix. $\newline$

Let $G$ be a bipartite graph with bipartition $(U,W)$, where
$U=\{u_1,\ldots,u_n\}$ and $W=\{w_1,\ldots,w_n\}$, and $M=\{u_i
w_i, 1\leq i \leq n\}$ a perfect matching of $G$. We form
$R(G)=(a_{ij})_{n\times n}$, where $a_{ij}=1$ if and only if $u_i
w_j\in E(G)$. Then $R(G)$ has a positive main diagonal, which
corresponds to $M$. We obtain a digraph $D=D(R(G)-I)$, where $I$
denote the identity matrix. On the contrary, given a digraph $D$,
we can get a bipartite graph $G=B(M(D)+I)$, which has a perfect
matching. Hence we have a corresponding between bipartite graphs
with a perfect matching and digraphs. We may get different $D$
from $G$, depending on how we choose the perfect matching $M$,
therefore we denote $D$ by $D=D(G,M)$. While $G$ is uniquely
determined by $D$, we denote it by $G=B(D)$. Clearly, such a
corresponding includes a bijection between $M$ and $V(D)$, and a
bijection between $E(G)\backslash M$ and $A(D)$. $D$ can also be
understood as obtained from $G$ by orienting all edges of $G$
towards the same partition and then contracting all edges of $M$.

There is a well-known equivalent property between the
1-extendibility of $G$ and the strong connectivity of $D$.
\begin{theorem}\label{theorem:lp1}(\cite{LP1}, Exercise 4.1.5)
Let $G$ be a bipartite graph and $M$ a perfect matching of $G$.
Then $D=D(G,M)$ is strong if and only if $G$ is 1-extendable.
\end{theorem}

The following is another interesting relation between $G$ and $D$.
\begin{theorem} \label{theorem:lp2}
(\cite{LP1}, Exercise 4.3.3) Let $G$ be a bipartite graph with a
unique perfect matching $M$. Then $D=D(G,M)$ is acyclic.
\end{theorem}

In this paper we further discuss the relation between $G$ and $D$,
as well as their relations with combinatorial matrices.

\section{Extendibility versus Connectivity}
Below is a generalization of Theorem \ref{theorem:lp1}, which has
been stated in \cite{Thm2006} without a proof.

\begin{theorem}\label{theorem:CeqE}
Let $G$ be a bipartite graph and $M$ a perfect matching of $G$.
Then $D=D(G,M)$ is $k$-strong if and only if $G$ is
$k$-extendable.
\end{theorem}

We prove Theorem \ref{theorem:CeqE} in this section and show some
interesting applications of it. We need Menger's Theorem in our
proof.

\begin{theorem} \label{theorem:_Menger} (Menger \cite{M3})
Let $D$ be a digraph. Then $D$ is $k$-strong if and only if
$|V(D)|\geq k+1$ and $D$ contains $k$ internally vertex disjoint
($s$,$t$)-paths for every choice of distinct vertices $s,t \in V$.
\end{theorem}

Actually we use an equivalent form of Menger's Theorem. Further
more, we only need the following weaken form, which appears as an
exercise in \cite{BG1}.

\begin{lemma}\label{lemma:weak_Menger}(\cite{BG1},~Exercise~7.17)
Let $D$ be a $k$-strong digraph. Let $x_1$, $x_2$, \ldots,
$x_{k-1}$, $y_1$, $y_2$, \ldots, $y_{k-1}$ be distinct vertices of
$D$, then there are $k$ independent paths in $D$, starting at
$x_i$, $0\leq i \leq k-1$ and ending at $y_j$, $0\leq j\leq k-1$.
\end{lemma}

Now comes the proof of Theorem \ref{theorem:CeqE}.

\begin{proof} Let $D$ be
$k$-strong. We use induction on $k$ to prove that $G$ is
$k$-extendable. When $k=1$, the conclusion follows from Theorem
\ref{theorem:lp1}. Suppose that the conclusion holds for all
integers $1\leq m < k$. Now we prove that an arbitrary matching
$M_0$ of size $k$ in $G$ is contained in a perfect matching of
$G$.

Firstly we assume that $|M_0 \cap M|\geq 1$. Let $e\in M_0 \cap M$
and the vertex in $D$ corresponding to $e$ be $v_e$. Let $G^\prime
= G-V(e)$, $D^\prime=D-v_e$, and $M^\prime=M\backslash {e}$. Then
$D^\prime$ is $(k-1)$-strong and $D^\prime=D(G^\prime,M^\prime)$.
By the induction hypothesis, $G^\prime$ is $(k-1)$-extendable.
Hence $M_0\backslash\{e\}$, which is a matching of size $k-1$ in
$G^\prime$, is contained in a perfect matching $M^\prime$ of
$G^\prime$. Then $M^\prime\cup \{e\}$ is a perfect matching of $G$
containing $M_0$.

Now we handle the case that $M_0 \cap M=\emptyset$. In this case,
$M_0$ corresponds to an arc set $A_0$ of order $k$ of $D$. The
arcs in $A_0$ form some independent cycles and paths in $D$. Let
the set of cycles formed be $\mathcal{C}_0=\{C_0,\ C_1,\ldots,\
C_{s-1}\}$ and the set of paths formed be $\mathcal{P}_0=\{P_0,\
P_1,\ldots,\ P_{t-1}\}$. Let the starting and ending vertices of
$P_i$ be $u_i$ and $v_i$, $0\leq i \leq t-1$. Let $V_0$ be the
union of the set of vertices of cycles in $\mathcal{C}_0$ and the
set of internal vertices of paths in $\mathcal{P}_0$. Then
$|V_0|=k-t$. By definition, $D-V_0$ is $t$-strong. By Lemma
\ref{lemma:weak_Menger}, there are $t$ independent paths in $D$
starting at $v_i$, $0\leq i\leq t-1$, and ending at $u_j$, $0\leq
j\leq t-1$. Such paths, together with the paths in
$\mathcal{P}_0$, form some independent cycles in $D$. Denote the
set of such cycles by $\mathcal{C}_1$. Then $\mathcal{C}_0\cup
\mathcal{C}_1$ is a set of independent cycles in $D$ which covers
all arcs in $A_0$. $\mathcal{C}_0\cup \mathcal{C}_1$ corresponds
to a set $\mathcal{C}$ of independent $M$-alternating cycles in
$G$. Let the set of edges of cycles in $\mathcal{C}$ be
$E(\mathcal{C})$, then $E(\mathcal{C})\vartriangle M$ is a perfect
matching of $G$ containing $M_0$. Hence $G$ is $k$-extendable.


Conversely, suppose that $G$ is $k$-extendable. To see that $D$ is
$k$-strong, let $\{v_1, v_2, \ldots, v_{k-1}\}$ be a set of $k-1$
vertices in $D$. Denote by $e_i$ the edge in $G$ corresponds to
$v_i$, $1\leq i\leq k-1$. Let $G^\prime=G-\cup^{k-1}_{i=1}V(e_i)$,
$D^\prime=D-\{v_i:1\leq i\leq k-1\}$ and $M^\prime=M\backslash
\{e_i:1\leq i\leq k-1\}$. Then $D^\prime=G(G^\prime,M^\prime)$.
Since $G$ is $k$-extendable, $G^\prime$ is 1-extendable. Hence
$D^\prime$ is strong by Theorem \ref{theorem:lp1} and $D$ is
$k$-strong.
\end{proof}

\begin{theorem}\label{theorem:minimal}
Let $G$ be a bipartite graph and $M$ a perfect matching of $G$. If
$G$ is minimal $k$-extendable then $D=D(G,M)$ is minimal
$k$-strong.
\end{theorem}
\begin{proof}
Suppose that $G$ is minimal $k$-extendable. By Theorem
\ref{theorem:CeqE}, $D$ is $k$-strong. Let $a$ be an arc of $D$
and $e$ be the edge corresponding to $a$ in $G$. Then
$D-a=D(G-e,M)$. By the minimality of $G$, $G-e$ is not
$k$-extendable, hence $D-a$ is not $k$-strong by Theorem
\ref{theorem:CeqE}. By the arbitrary of $a$, $D$ is minimal
$k$-strong.
\end{proof}

The converse of Theorem \ref{theorem:minimal} does not generally
hold, that is, $G$ does not need to be minimal $k$-extendable if
$D=D(G,M)$ is minimal $k$-strong. For example, we show a minimal
strong digraph $D_0$ in Figure \ref{fig:digraph} and $G_0=B(D_0)$,
which is not minimal 1-extendable, in Figure \ref{fig:graph}.

\begin{figure}[htbp]
\begin{minipage}[t]{0.5\linewidth}
\centering
\includegraphics[width=0.8\linewidth]{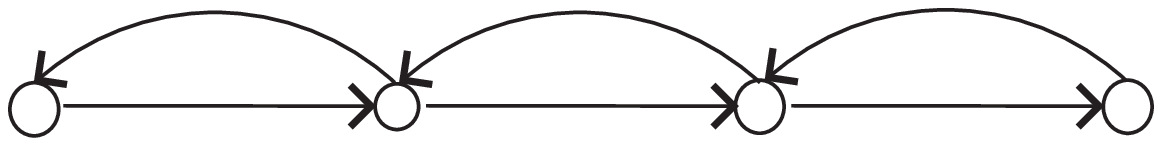}
\caption{A minimal strong digraph $D_0$} \label{fig:digraph}
\end{minipage}
\begin{minipage}[t]{0.5\linewidth}
\centering
\includegraphics[width=0.8\linewidth]{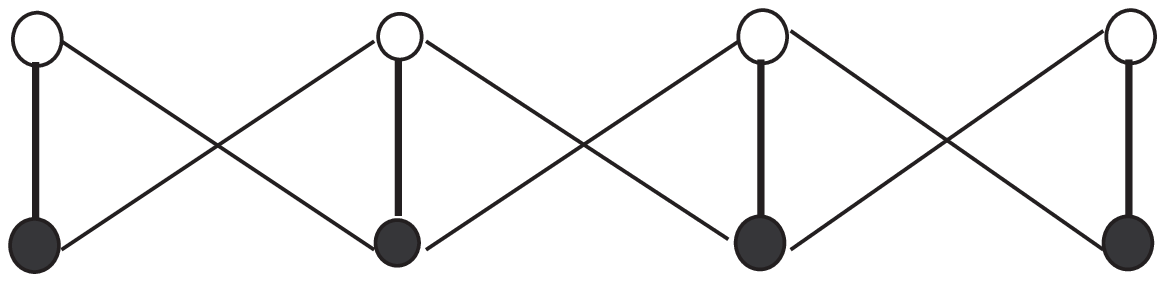}\label{fig:graph}
\caption{$G_0=B(D_0)$}
\end{minipage}
\end{figure}

$\newline$

There are many parallel results on $k$-extendable bipartite graphs
and $k$-strong digraphs. Theorem \ref{theorem:CeqE} and Theorem
\ref{theorem:minimal} help to explain such a similarity between
these two classes of graphs. In the rest of this section, we will
illustrate some such results.

Our first demonstrations are the well-known ear decompositions of
strong digraphs and 1-extendable bipartite graphs.

An \emph{ear decomposition} of a digraph $D$ is a sequence
$\mathcal{E}= \{P_0, P_1, ..., P_t\}$, where $P_0$ is a cycle and
each $P_i$ is a path, or a cycle with the following properties:
$\newline$(a) $P_i$ and $P_j$ are arc disjoint when $i\neq j$.
$\newline$(b) For each $i=1,...,t$, if $P_i$ is a cycle, then it
has precisely one vertex in common with $V(D_{i-1})$. Otherwise
the end-vertices of $P_i$ are distinct vertices of $V(D_{i-1})$
and no other vertex of $P_i$ belongs to $V(D_{i-1})$. Here $D_i$
denotes the digraph with vertices $\bigcup _{j=0} ^{i} V(P_j)$ and
arcs $\bigcup _{j=0} ^{i} A(P_j)$.
$\newline$ (c) $\bigcup _{j=0} ^{t} V(P_j)=V(D)$ and $\bigcup
_{j=0} ^{t} A(P_j)=A(D)$.

\begin{theorem} (\cite{BG1}, Theorem 7.2.2)
A digraph is strong if and only if it has an ear decomposition.
Furthermore, if $D$ is strong, then for every vertex $v$, every
cycle $C$ containing $v$ can be used as starting cycle $P_0$ for
an ear decomposition of $D$.
\end{theorem}

Let $e$ be an edge and $G_0$ be the graph containing $e$ only.
Join the end-vertices of $e$ by an odd path $P_1$ we obtain a
graph $G_1$. Now if $G_{i-1}=e+P_1+\ldots+P_{i-1}$ has already
been constructed, join any two vertices in different color classes
of $G_{i-1}$ by an odd path $P_i$ having no other vertices in
common with $G_{i-1}$ we obtain $G_i$. The decomposition
$G_r=e+P_1+\ldots+P_r$ is called a \emph{bipartite ear
decomposition} of $G_r$.
\begin{theorem} (\cite{LP1}, Theorem 4.1.6)
A bipartite graph is 1-extendable if and only if it has a
bipartite ear decomposition. Such an ear decomposition may be
started with any edge $e$ of $G$.
\end{theorem}

It is remarked in \cite{LP1} that, given a bipartite ear
decomposition $G=e+P_1+\ldots+P_r$ of a bipartite graph $G$, there
is exactly one perfect matching $M$ in $G$ such that $M\cap
E(G_i)$ is a perfect matching of $G_i$ for every $i$, $0\leq i
\leq r$. It is not hard to check that the given bipartite ear
decomposition corresponds to an ear decomposition of the digraph
$D=D(G,M)$.$\newline$

Next, we show two corresponding characterizations.
\begin{theorem} \label{theorem:P1}(Plummer \cite{P1})
Let $G$ be a connected bipartite graph with bipartition $(U,W)$,
$k$ a positive integer such that $k\leq (|V(G)|-2)/2$. Then $G$ is
$k$-extendable if and only if $|U|=|W|$ and for all non-empty
subset $X$ of $U$ with $|X|\leq |U|-k$, $|N(X)|\geq |X|+k$.
\end{theorem}

Let $D$ be a digraph. Let $X$, $Y$ be disjoint non-empty proper
subsets of $V(D)$, the ordered pair $(X,Y)$ is called a
\emph{one-way pair} in $D$ if $D$ has no arc with tail in $X$ and
head in $Y$. Let $h(X,Y)=|V-X-Y|$.

\begin{theorem} \label{theorem:FJ1}(Frank and Jord\'{a}n \cite{FJ1})
A digraph $D$ is $k$-strong if and only if $h(X,Y)\geq k$ for
every one-way pair $(X,Y)$ in $D$.
\end{theorem}

The condition in Theorem \ref{theorem:FJ1} is equivalent to that
$N^+(X)\geq k$ for any set $X\subseteq V(D)$ with $|X|\leq
|V(D)|-k$, which is similar to the condition in Theorem
\ref{theorem:P1}.$\newline$

The counterpart of Menger's Theorem for bipartite $k$-extendable
graphs was proved by Aldred et al. in \cite{AHLS}. The original
proof is a little involved. Now, with Theorem \ref{theorem:CeqE},
we can deduce it from Menger's Theorem straightly.
\begin{theorem} \label{theorem:A1}
Let $G$ be a bipartite graph with bipartition ($U$,$W$) and a
perfect matching. Then $G$ is $k$-extendable if and only if for
any perfect matching $M$ and for each pair of vertices $u\in U$
and $w \in W$, there are $k$ internally disjoint $M$-alternating
paths connecting $u$ and $w$, furthermore, these $k$ paths start
and end with edges in $E(G)\backslash M$.
\end{theorem}
\begin{proof}
Let $M$ be any perfect matching of $G$, and $D=D(G,M)$ be obtained
by orienting all edges of $G$ towards $W$ then contracting all
edges in $M$. Suppose that $G$ is $k$-extendable. Firstly we prove
the below claim.
$\newline$\textbf{Claim 1}. Let $D$ be a $k$-strong digraph and
$x$ a vertex of $D$, then $D$ contains $k$ cycles, any two of
which intersect at $x$ only.
\begin{proof}
Let $x^\prime$ be a vertex not in $V(D)$. Construct $D^\prime$
such that $V(D^\prime)=V(D)\cup \{x^\prime\}$,
$A(D^\prime)=A(D)\cup\{ux^\prime:ux\in A(D)\}\cup\{x^\prime
u:xu\in A(D)\}$. We prove that $D^\prime$ is $k$-strong. If
$D^\prime$ is not $k$-strong, then there exists a separator $S$ of
size less than $k$. If $S$ contains $x^\prime$, then $S-x^\prime$
is a separator of $D$ of size less than $k-1$, contradicting the
strong connectivity of $D$. Assume that $S$ does not contain
$x^\prime$, then any vertex $y$ which is separated from $x^\prime$
by $S$ is separated from $x$ by $S$ as well, hence $S$ is a
separator of $D$, again contradicting the strong connectivity of
$D$. Therefore $D^\prime$ is $k$-strong. By Menger's Theorem there
are $k$ internally disjoint $(x,x^\prime)$-paths in $D$. Replacing
every arc $ux^\prime$ in these paths with the arc $ux$, we obtain
the cycles as claimed.
\end{proof}

By Theorem \ref{theorem:CeqE}, $D$ is $k$-strong. If $uw\notin M$,
let $uu^\prime, w^\prime w \in M$, and $u_0, w_0 \in V(D)$ be the
vertices of $D$ corresponding to edges $uu^\prime$ and
$ww^\prime$. By Menger's Theorem there are $k$ internally disjoint
paths in $D$ from $u_0$ to $w_0$, which correspond to $k$
$M$-alternating paths in $G$ from $u^\prime$ to $w^\prime$,
starting and ending with the edges $u^\prime u$ and $ww^\prime$,
respectively. Furthermore, any two of these $M$-alternating paths
intersect at the edges $u^\prime u$ and $ww^\prime$ only. Removing
$u^\prime u$ and $ww^\prime$ from these paths we obtain $k$
internally disjoint $M$-alternating paths from $u$ to $w$ in $G$,
starting and ending with edges in $E(G)\backslash M$. If $uw\in
M$, let $v \in V(D)$ be the vertices of $D$ corresponding to $uw$.
By Claim 1 there are $k$ cycles in $D$, any two of which intersect
at $v$ only. The cycles correspond to $k$ $M$-alternating cycles
in $G$, any two of which intersect at the edge $uw$ only. Removing
$uw$ from the cycles we obtain the paths we want.

Conversely, suppose that for $M$, any vertices $u$ and $w$ in $G$,
we can always find the $M$-alternating paths as stated. Let $v_1$,
$v_2$ be any two vertices in $D$ and $u_1w_1$, $u_2w_2$ be the
edges in $M$ corresponding to $v_1$ and $v_2$, where $u_i \in U$
and $w_i \in W$, $i=1,2$. Then there are $k$ internally disjoint
$M$-alternating paths from $u_1$ to $w_2$, starting and ending
with edges in $E(G)\backslash M$. Adding edges $u_1w_1$ and
$u_2w_2$ to each of the paths, we get $k$ $M$-alternating paths,
corresponding to $k$ internally disjoint paths in $D$ from $v_1$
to $v_2$. Since $v_1$, $v_2$ is arbitrarily chosen, by Menger's
Theorem, $D$ is $k$-strong. By Theorem \ref{theorem:CeqE}, $G$ is
$k$-extendable.
\end{proof}

When considering minimal $k$-extendable bipartite graph and
minimal $k$-strong digraphs, We find the following similar
results.

\begin{theorem} \label{theorem:Mader1} (Mader \cite{M1})
Every minimal $k$-strong digraph contains at least $k$ vertices of
out-degree $k$ and at least $k$ vertices of in-degree $k$.
\end{theorem}

\begin{theorem} \label{theorem:Lou1} (Lou \cite{Lou1})
Every minimal $k$-extendable bipartite graph $G$ with bipartition
($U$, $W$) has at least $2k + 2$ vertices of degree $k + 1$.
Furthermore, both $U$ and $W$ contain at least $k + 1$ vertices of
degree $k + 1$.
\end{theorem}

Neither of them implies the other but striking analogical
techniques were used in \cite{M1} and \cite{Lou1}. We cite two
corresponding structural lemmas here.

Let $h(a)$ and $t(a)$ denote the head and tail of an arc $a$,
respectively. An arc set $a_1$, $a_2$, $\ldots$, $a_m$, where $m$
is even, is call an \emph{anti-directed trail} if for all $i$,
$h(a_{2i+1})=h(a_{2i+2})$ and $t(a_{2i+2})=t(a_{2i+3})$, or for
all $i$, $t(a_{2i+1})=t(a_{2i+2})$ and $h(a_{2i+2})=h(a_{2i+3})$
(indexes modula $m$).

\begin{theorem} \label{theorem:Mader2} (Mader \cite{M1})
Let $D$ be a minimal $k$-strong digraph. Then the subgraph of $D$
induced by all arcs whose tail is of outdegree at least $k+1$ and
whose head is of indegree at least $k+1$ does not contain an
anti-directed trail.
\end{theorem}

\begin{theorem} \label{theorem:Lou2} (Lou \cite{Lou1})
In a minimal $k$-extendable bipartite graph, the subgraph induced
by the edges both ends of which have degree at least $k$+2 is a
forest.
\end{theorem}
It can be verified that an anti-directed trail in $D$ corresponds
to a closed trail in $G=B(D)$, while a closed trail in $G$ does
not always corresponds to an anti-directed trail in $D$.
\section{Combinatorial Matrices}
In this section, we show the equivalence among $k$-connected
digraphs, $k$-extendable bipartite graphs and combinatorial
matrices.

\begin{theorem} (\cite {LL2}, Theorem 2.1.1) \label{theorem:Liu_irrd}
Let $A \in B_n$, then $A$ is irreducible if and only if the
associated digraph $D(A)$ is strong.
\end{theorem}

\begin{theorem} (Brualdi et al. \cite{BPS1}) \label{theorem:Brualdi_indecomp}
Let $A\in B_n$. Then $A$ is fully indecomposable if and only if
every one entry of $A$ lies in a nonzero diagonal, and every zero
entry of $A$ lies in a diagonal with exactly one zero member.
\end{theorem}
A nonzero diagonal of $A$ corresponds to a perfect matching of the
reduced associated bipartite graph $B(A)$. The condition in
Theorem \ref{theorem:Brualdi_indecomp} is equivalent to that
$B(A)$ is 1-extendable.

\begin{theorem} (\cite{LL2}, Theorem 2.1.3) \label{theorem:Liu_irrd_indecomp}
Let $A\in B_n$, Then

(1) If $A$ is fully indecomposable, then $A$ is irreducible.

(2) A is irreducible if and only if $A+I$ is fully indecomposable.
\end{theorem}

The followings are generalized results for $k$-indecomposable
matrices and $k$-irreducible matrices.

\begin{theorem} \label{theorem:You:irrd_str} (You et al. \cite{YLS1})
Suppose $k\geq 1$. Then a matrix $A\in B_n$ is $k$-irreducible if
and only if $D(A)$ is $k$-strong.
\end{theorem}

\begin{theorem} \label{theorem:Ind=Ext}
Suppose $0 \leq k \leq n-1$ and $A\in B_n$. Then $A$ is
$k$-indecomposable if and only if $G=B(A)$ is $k$-extendable.
\end{theorem}
\begin{proof}
Suppose that $A$ is $k$-indecomposable. Let the bipartition of $G$
be $(U,W)$. Let $U_1$ be a subset of $U$ such that $|U_1|\leq
n-k$. If $|N(U_1)|\leq |U_1|+k-1$, then $|W\backslash N(U_1)|\geq
n-|U_1|-k+1$, and the submatrix of $A$ indexed by $U_1$ and
$W\backslash N(U_1)$ is a zero matrix of size at least
$|U_1|\times (n-k+1-|U_1|)$. By definition, $A$ is $k$-partly
decomposable, a contradiction. Hence $|N(U_1)|\geq |U_1|+k$. By
Theorem \ref{theorem:P1}, $G$ is $k$-extendable.

Conversely, suppose that $G$ is $k$-extendable. If $A$ is
$k$-partly decomposable then $A$ has an $l\times (n-k+1-l)$ zero
submatrix, for some $1\leq l \leq n-k$. Let the subset of $V(G)$
indexing the row of the submatrix be $U_1$, then $|U_1|=l \leq
n-k$ and $|N(U_1)|\leq n-(n-k+1-l)=l+k-1=|U_1|+k-1$, contradicting
Theorem \ref{theorem:P1}.
\end{proof}

\begin{lemma} \label{theorem:You:inde_str}(You et al. \cite{YLS1})
Suppose $k\geq 1$  and $A \in B_n$ has a positive main diagonal.
Then A is k-indecomposable if and only if $D(A)$ is $k$-strong.
\end{lemma}

\begin{theorem} \label{theorem:inde_irre}
Let $A\in B_n$, Then

(1) If $A$ is $k$-indecomposable, then $A$ is $k$-irreducible.

(2) $A$ is $k$-irreducible if and only if $A+I$ is
$k$-indecomposable.
\end{theorem}
\begin{proof}
By definition, if $A$ is $k$-reducible then $A$ is
$k$-decomposable. Hence if $A$ is $k$-indecomposable, $A$ is
$k$-irreducible and (1) holds.

By Theorem \ref{theorem:You:irrd_str}, $A$ is irreducible if and
only if $D(A)$ is $k$-strong. Since adding a loop to a vertex or
removing a loop from a vertex does not affect the strongly
connectivity of a digraph, $D(A)$ is $k$-strong if and only if
$D(A+I)$ is $k$-strong. By Lemma \ref{theorem:You:inde_str},
$D(A+I)$ is $k$-strong if and only if $A+I$ is $k$-indecomposable.
\end{proof}

\section{Elementary components versus strong components}

Let $G$ be a bipartite graph with a perfect matching $M$, but not
1-extendable. By Theorem \ref{theorem:CeqE}, $D=D(G,M)$ is not
strong. In this section, we consider the elementary components of
$G$ and the strong components of $D$.
\begin{lemma} \label{lemma:ele_matching}
Let $G$ be a bipartite graph with a perfect matching $M$, and
$G_1$ an elementary component of $G$, then $E(G_1)\cap M$ is a
perfect matching of $G_1$.
\end{lemma}
\begin{proof}
An edge $e\in E(G)\backslash E(G_1)$ incident to a vertex in $G_1$
is fixed. However it can not be a fixed double edge, since every
edge adjacent to a fixed double edge must be a fixed single edge.
Hence, all edges in $M$ saturating vertices in $V(G_1)$ must be in
$E(G_1)$ and $E(G_1)\cap M$ is a perfect matching of $G_1$.
\end{proof}

Let $M_1=E(G_1)\cap M$, then $D_1=D(G_1,M_1)$ is a subdigraph of
$D$. Moreover, let $G_1$ be a subgraph of $G$ consisting of only a
fixed double $e$ edge of $G$, then $e\in M$ and $D_1=D(G_1,\{e\})
$ contains only one vertex of $D$.
\begin{theorem} \label{theorem:components}
Let $G$ be a bipartite graph with a perfect matching $M$, $G_1$ a
subgraph of $G$ such that $M_1=E(G_1)\cap M$ is a perfect matching
of $G_1$. Let $D=D(G,M)$ and $D_1=D(G_1,M_1)$. Then the followings
are equivalent.

(1) $G_1$ is an elementary component of $G$, or consists of a
fixed double edge only.

(2) $D_1$ is a strong component of $D$.
\end{theorem}
\begin{proof}
Suppose that $G_1$ is an elementary component of $G$. Then $G_1$
is 1-extendable and hence $D_1$ is strong. Assume that $D_1$ is
properly contained in a strong subdigraph $D_1^\prime$ of $D$.
Then $G_1^\prime=B(D_1^\prime)$ is a 1-extendable subgraph of $G$
containing $G_1$. Furthermore, any perfect matching of
$G_1^\prime$ is contained in a perfect matching of $G$. Therefore
any edge of $G_1^\prime$ is contained in a perfect matching of
$G$. However any edge in $E(G_1^\prime)\backslash E(G_1)$ incident
to a vertex of $G_1$ must be a fixed single edge and can not be
contained in any perfect matching of $G$, which leads to a
contradiction. Hence $D_1$ is a maximal strong subdigraph, that
is, a strong component, of $D$.

Suppose that $G_1$ consists of a fixed double edge $e$ only. Then
$e\in M$ and $D(G_1,\{e\})$ contains exactly a vertex $v$ in $D$.
If $v$ is properly contained in a strong component $D_1^\prime$ of
$D$, then $G_1^\prime=B(D_1^\prime)$ is a 1-extendable subgraph of
$G$ containing $e$. Furthermore, every perfect matching of
$G_1^\prime$ is contained in a perfect matching of $G$. Hence
every edge of $G_1^\prime$ is contained in a perfect matching of
$G$. However $e$ is contained in every perfect matching of $G$, so
all edges adjacent to $e$ are fixed single edges and cannot be
contained in a perfect matching of $G$, a contradiction. Hence $v$
composes a strong component of $D$ with only one vertex.

Conversely, let $D_1$ be a strong component of $D$. Then
$G_1=B(D_1)$ is 1-extendable. To prove that $G_1$ is an elementary
component or consist of a fixed double edge, we need only to prove
that an edge $e=u_1u_2\in E(G)\backslash E(G_1)$ associated with a
vertex $u_1\in V(G_1)$ is a fixed single edge. Suppose that $e$ is
not a fixed single edge and contained in a perfect matching
$M^\prime$ of $G$. Let $u_1w_1$, $u_2w_2\in M$, which correspond
to vertices $v_1$ and $v_2$ in $D$ respectively, then $v_1\in
V(D_1)$ and $v_2\notin V(D_1)$. $M\vartriangle M^\prime$ consists
of nonadjacent edges and alternating cycles. The edges $e$,
$u_1w_1$ and $u_2w_2$ must be contained in an alternating cycle
$C$. However $C$ corresponds to a directed cycle in $D$, which
contains $v_1$ and $v_2$. This contradicts the fact that $D_1$ is
a strong component of $D$.
\end{proof}


\end{document}